\documentclass[12pt]{amsart}

\usepackage{amssymb}

\usepackage{enumitem}

\usepackage{graphicx}

\makeatletter
\@namedef{subjclassname@2020}{%
  \textup{2020} Mathematics Subject Classification}
\makeatother

\newtheorem{theorem}{Theorem}[section]

\newtheorem{lemma}[theorem]{Lemma}
\newtheorem{proposition}[theorem]{Proposition}

\theoremstyle{definition}
\newtheorem{definition}[theorem]{Definition}
\newtheorem{remark}[theorem]{Remark}

\numberwithin{equation}{section}

\begin{document}

\baselineskip=17pt

%%%%%%%%%%%%%%%%

\title[Boundedness on weighted flag and product Hardy spaces]{The complete boundedness of singular integrals on weighted flag and product Hardy spaces}

\author[J. Tan]{Jian Tan}
\address{School of Science\\ Nanjing University of Posts and Telecommunications\\
    Nanjing 210023, People's Republic of China}
\email{tj@njupt.edu.cn}

\date{}

\begin{abstract}
It is known that 
product singular integrals are bounded on product Hardy spaces and that
flag singular integrals are bounded on flag Hardy spaces. The purpose of this paper is to obtain the complete boundedness of singular integrals on weighted flag and product Hardy spaces. In particular, we prove the boundedness of one-parameter singular integrals on both weighted flag and product Hardy spaces, as well as the boundedness of flag singular integrals on weighted product Hardy spaces.

%including the boundedness of one-parameter singular integrals on weighted flag and product Hardy spaces and the boundedness of flag singular integrals on weighted product Hardy spaces.
\end{abstract}



\maketitle

\section{Introduction and statement of main results}

It is well known that the classical singular integral operators
are bounded on the classical Hardy spaces.
On the other hand, the product singular integral operators, which are invariant under the following dilations:
$$
\delta x=(\delta_1x_1,\delta_2x_2,\cdots,\delta_nx_n),
$$
with $\delta_j>0$ and $1\le j\le n$,
are also bounded on corresponding product Hardy spaces.
For more details, see for example \cite{CF,F,FS,GS,HLLL} and the references therein.
A new extension of product theory came to light with the proof by
M\"uller et al. of the $L^p$ boundedness when $p\in(1,\infty)$
of Marcinkiewicz multipliers on the Heisenberg group in \cite{MRS}.
Then Han et al. introduced the multi-parameter flag Hardy spaces and gave the boundedness of flag singular integrals on these multi-parameter flag Hardy spaces in \cite{HLLW,HL,HLS}, which can be viewed as an ``intermediate'' between the classical Hardy spaces and the product Hardy spaces.
A natural question arises: is it possible that the one-parameter singular integrals are bounded on the flag Hardy spaces?

One of the main goals of this paper is to address this question, focusing on the two cases of $\mathbb R^n\times \mathbb R^m$ and $\mathbb R^n\times \mathbb R^m
\times \mathbb R^d$ with the flag structures, which also demonstrates that our methods are suitable for cases involving arbitrary multi-parameter flag structures.
In this paper, first we will consider the following one-parameter singular integral operators on $\mathbb R^n\times \mathbb R^m$.

\begin{definition}\label{1-2}
${K} \in C^2(\mathbb{R}^{n+m} \setminus \{0\})$ is said to be a {\it one-parameter singular integral kernel} on $\mathbb R^n\times \mathbb R^m$ if there are constants $C$ such that
\[
|\partial_z^\alpha {K}(z)| \leq \frac{C}{|z|^{n+m+|\alpha|}}
\]
for $0 \leq |\alpha| \leq 2$ and $z \in \mathbb{R}^{n+m} \setminus \{0\}$, and
\[
\int_{\epsilon < |z| < N} {K}(z) \, dz \leq C,
\]
whenever $0 < \epsilon < N < \infty$.
We say that an operator $T$ is a {\it one-parameter singular integral operator} if $Tf(z) = p. v. K\ast f(z)$ for a singular integral kernel ${K}$.
\end{definition}

For convenience, we defer the definitions of multi-parameter weighted flag and product Hardy spaces to Section \ref{pre}, along with Muckenhoupt weights in the multi-parameter settings.
Hereafter denote $\tau=\frac{2n}{2n+1}\vee\frac{2m}{2m+1},$ where $a\vee b$ means $\max\{a,b\}$. The first result in our paper is as follows. 
\begin{theorem}\label{1f}
Suppose that $w\in A_{p/\tau}(\mathbb R^n\times \mathbb R^m)$. If $T$ is a one-parameter singular integral operator on $\mathbb R^n\times \mathbb R^m$, then $T$ is bounded on multi-parameter weighted flag Hardy spaces $H^p_{flag,w}(\mathbb R^n\times \mathbb R^m)$ for $\tau<p\le1$.
\end{theorem}

In \cite{Tanc}, C. Tan proved that one-parameter singular integral operators
are also bounded on the product Hardy spaces. In the following theorem we extend the result to
the weighted case. Denote $\tilde\tau=\frac{n}{n+1}\vee\frac{m}{m+1}.$

\begin{theorem}\label{1p}
Suppose that $w\in A_{p/{\tilde\tau}}(\mathbb R^n\times \mathbb R^m)$. If $T$ is a one-parameter singular integral operator on $\mathbb R^n\times \mathbb R^m$, then $T$ is bounded on weighted product Hardy spaces $H^p_{prod,w}(\mathbb R^n\times \mathbb R^m)$ for $\tilde\tau<p\le1$.
\end{theorem}

\begin{remark}\label{r1}
If the smoothness conditions of $K$ are replaced by ${K} \in C^\infty(\mathbb{R}^{n+m} \setminus \{0\})$, 
\[
|\partial_z^\alpha {K}(z)| \leq \frac{C}{|z|^{n+m+|\alpha|}}
\]
for any $0 \leq |\alpha| <\infty$ and
$
\int_{\epsilon < |z| < N} \partial_z^\beta{K}(z) \, dz \leq C
$
for any $0 \leq |\beta| <\infty$ and $0 < \epsilon < N < \infty$,
then by using the extrapolation we can get that for any $w\in A_{\infty}(\mathbb R^n\times \mathbb R^m)$, the one-parameter singular integral operator $\mathcal T$ is bounded on $H^p_{flag,w}(\mathbb R^n\times \mathbb R^m)$ and $H^p_{prod,w}(\mathbb R^n\times \mathbb R^m)$ for $0<p<\infty$.
\end{remark}

We recall the definition of flag singular integrals in \cite{MRS}.

\begin{definition}\label{flag}
A \emph{flag kernel} on $\mathbb{R}^n \times \mathbb{R}^m$ is a distribution on $\mathbb{R}^{n+m}$ which coincides with a $C^{\infty}$ function away from the coordinate subspace $\{(0, y)\} \subset \mathbb{R}^{n+m}$, where $0 \in \mathbb{R}^n$ and $y \in \mathbb{R}^m$, and satisfies the following conditions:

\begin{itemize}
    \item[(1)] (Differential Inequalities) For any multi-indices $\alpha = (\alpha_1, \cdots, \alpha_n)$ and $\beta = (\beta_1, \cdots, \beta_m)$, we have
    \[
    |\partial_x^{\alpha} \partial_y^{\beta} K(x, y)| \leq C_{\alpha, \beta} |x|^{-n - |\alpha|} \cdot (|x| + |y|)^{-m - |\beta|}
    \]
    for all $(x, y) \in \mathbb{R}^n \times \mathbb{R}^m$ with $|x| \neq 0$.

    \item[(2)] (Cancellation Condition)
    \[
    \left| \int_{\mathbb{R}^m} \partial_x^{\alpha} K(x, y) \phi_1(\delta y) \, dy \right| \leq C_{\alpha} |x|^{-n - |\alpha|}
    \]
    for every multi-index $\alpha$, every normalized bump function $\phi_1$ on $\mathbb{R}^m$, and every $\delta > 0$;

    \[
    \left| \int_{\mathbb{R}^n} \partial_y^{\beta} K(x, y) \phi_2(\delta x) \, dx \right| \leq C_{\beta} |y|^{-m - |\beta|}
    \]
    for every multi-index $\beta$, every normalized bump function $\phi_2$ on $\mathbb{R}^n$, and every $\delta > 0$;

    \[
    \left| \int_{\mathbb{R}^{n+m}} K(x, y) \phi_3(\delta_1 x, \delta_2 y) \, dx \, dy \right| \leq C
    \]
    for every normalized bump function $\phi_3$ on $\mathbb{R}^{n+m}$, and every $\delta_1 > 0$ and $\delta_2 > 0$.
\end{itemize}

Then $T_{flag}(f)(x, y)=K\ast f(x, y)$ is said to be a {\it flag singular integral} on $\mathbb R^n\times \mathbb R^m$.
\end{definition}

In \cite{HL}, Han and Lu obtained the flag singular integrals $T_{flag}$ are bounded on the 
flag Hardy spaces $H^p_{flag}(\mathbb R^n\times \mathbb R^m).$  
Thus, a natural question arises: are the flag singular integrals bounded on weighted product Hardy spaces on $\mathbb R^n\times \mathbb R^m$?
The another purpose of this paper is to answer in the affirmative by obtaining the boundedness of such operators on these spaces as follows.

\begin{theorem}\label{fp}
Suppose that $w\in A_{\infty}(\mathbb R^n\times \mathbb R^m)$. If $T_{flag}$ is a a flag singular integral on $\mathbb R^n\times \mathbb R^m$, then $T_{flag}$ is bounded on weighted product Hardy spaces $H^p_{prod,w}(\mathbb R^n\times \mathbb R^m)$ for $0<p\le1$.
\end{theorem}

We next study the weighted flag and product Hardy spaces boundedness of one-parameter singular integral operators on $\mathbb R^n\times \mathbb R^m\times \mathbb R^d$ which is nearly identical to Definition \ref{1-2}. For completeness, we give the definition of one-parameter singular integral operators on $\mathbb R^n\times \mathbb R^m\times \mathbb R^d$.

\begin{definition}\label{1-3}
${K} \in C^3(\mathbb{R}^{n+m+d} \setminus \{0\})$ is said to be a {\it one-parameter singular integral kernel} on $\mathbb R^n\times \mathbb R^m\times \mathbb R^d$ if there are constants $C$ such that
\[
|\partial_z^\alpha {K}(z)| \leq \frac{C}{|z|^{n+m+d+|\alpha|}}
\]
for $0 \leq |\alpha| \leq 3$ and $z \in \mathbb{R}^{n+m+d} \setminus \{0\}$, and
\[
\int_{\epsilon < |z| < N} {K}(z) \, dz \leq C,
\]
whenever $0 < \epsilon < N < \infty$.
We say that an operator $T$ is a {\it one-parameter singular integral operator} if $Tf(z) = p.v.{K}\ast f(z)$ for a singular integral kernel ${K}$.
\end{definition}

Hereafter denote $\gamma=\frac{3n}{3n+1}\vee\frac{3m}{3m+1}\vee \frac{3d}{3d+1}$. 
Now we state the following theorem.
\begin{theorem}\label{1f-3}
Suppose that $w\in A_{p/\gamma}(\mathbb R^n\times \mathbb R^m\times \mathbb R^d)$. If $T$ is a one-parameter singular integral operator on $\mathbb R^n\times \mathbb R^m\times \mathbb R^d$, then $T$ is bounded on multi-parameter weighted flag Hardy spaces $H^p_{flag,w}(\mathbb R^n\times \mathbb R^m\times \mathbb R^d)$ for $\gamma<p\le1$.
\end{theorem}

The boundedness of one-parameter singular integral operator is also established on
weighted product Hardy spaces $H^p_{prod,w}(\mathbb R^n\times \mathbb R^m\times \mathbb R^d)$ as follows. Denote $\tilde \gamma=\frac{n}{n+1}\vee\frac{m}{m+1}\vee \frac{d}{d+1}$.

\begin{theorem}\label{1p-3}
Suppose that $w\in A_{p/\tilde\gamma}(\mathbb R^n\times \mathbb R^m\times \mathbb R^d)$. If $T$ is a one-parameter singular integral operator on $\mathbb R^n\times \mathbb R^m\times \mathbb R^d$, then $T$ is bounded on weighted product Hardy spaces $H^p_{prod,w}(\mathbb R^n\times \mathbb R^m\times \mathbb R^d)$ for $\tilde\gamma<p\le1$.
\end{theorem}

\begin{remark}\label{r2}
If we impose stronger smoothness conditions of $K$ as in Remark \ref{r1},
then we can similarly conclude that for any $w\in A_{\infty}(\mathbb R^n\times \mathbb R^m\times \mathbb R^d)$, the operator $\mathcal T$ is bounded on $H^p_{flag,w}(\mathbb R^n\times \mathbb R^m\times \mathbb R^d)$ and $H^p_{prod,w}(\mathbb R^n\times \mathbb R^m\times \mathbb R^d)$ for $0<p<\infty$.
\end{remark}

Next we will consider the singular integrals with the following flag
kernels which can be found in \cite{HLW}.

\begin{definition}
A {\it flag kernel} is a distribution $\mathcal{K}$ on $\mathbb R^n\times \mathbb R^m\times \mathbb R^d$ which coincides with a $C^\infty$ function away from the coordinate subspace $x_1 = 0$ and satisfies the following:

\begin{itemize}
    \item[(i)] ({differential inequalities}) For each $\alpha = (\alpha_1, \alpha_2, \alpha_3) \in \mathbb{Z}^3$,
    \[
    \left|\partial_{x_1}^{\alpha_1} \partial_{x_2}^{\alpha_2} \partial_{x_3}^{\alpha_3} \mathcal{K}(x)\right| \lesssim |x_1|^{-n_1 - |\alpha_1|} (|x_1| + |x_2|)^{-n_2 - |\alpha_2|} (|x_1| + |x_2| + |x_3|)^{-n_3 - |\alpha_3|}
    \]
    for $x_1 \neq 0$.

    \item[(ii)] ({cancellation conditions})
    \begin{itemize}
        \item[(a)] Given normalized bump functions $\psi_i$, $i = 1, 2, 3$, on $\mathbb{R}^{n_i}$ and any scaling parameter $r > 0$, define a distribution $\mathcal{K}_{\psi_i, r}$ by setting
        \begin{equation} \label{1-3}
        \langle \mathcal{K}_{\psi_i, r}, \varphi \rangle = \langle \mathcal{K}, (\psi_i)_r \otimes \varphi \rangle
        \end{equation}
        for any test function $\varphi \in \mathcal{S}(\mathbb{R}^{N - n_i})$. Then the distributions $\mathcal{K}_{\psi_i, r}$ satisfy the differential inequalities
        \[
        \left|\partial_{x_2}^{\alpha_2} \partial_{x_3}^{\alpha_3} \mathcal{K}_{\psi_1, r}(x_2, x_3)\right| \lesssim |x_2|^{-n_2 - |\alpha_2|} (|x_2| + |x_3|)^{-n_3 - |\alpha_3|},
        \]
        \[
        \left|\partial_{x_1}^{\alpha_1} \partial_{x_3}^{\alpha_3} \mathcal{K}_{\psi_2, r}(x_1, x_3)\right| \lesssim |x_1|^{-n_1 - |\alpha_1|} (|x_1| + |x_3|)^{-n_3 - |\alpha_3|},
        \]
        \[
        \left|\partial_{x_1}^{\alpha_1} \partial_{x_2}^{\alpha_2} \mathcal{K}_{\psi_3, r}(x_1, x_2)\right| \lesssim |x_1|^{-n_1 - |\alpha_1|} (|x_1| + |x_2|)^{-n_2 - |\alpha_2|}.
        \]
    \item[(b)] For any bump functions $\psi_i$ on $\mathbb{R}^{N - n_i}$ and any parameters $r = (r_1, r_2)$, we define the distributions $\mathcal{K}_{\psi_i, r}$ by \eqref{1-3}. Then the distributions $\mathcal{K}_{\psi_i, r}$, $i = 1, 2, 3$, are one-parameter kernels and satisfy
    \[
    \left|\partial_{x_i}^{\alpha_i} \mathcal{K}_{\psi_i, r}(x_i)\right| \lesssim |x_i|^{-n_i - |\alpha_i|}.
    \]

    \item[(c)] For any bump function $\psi$ on $\mathbb{R}^N$ and $r_1, r_2, r_3 > 0$, we have
    \[
    |\langle \mathcal{K}, \psi(r_1 \cdot, r_2 \cdot, r_3 \cdot) \rangle| \lesssim 1.
    \]
        \end{itemize}
\end{itemize}

Moreover, the corresponding constants that appear in these differential inequalities are independent of $r$, $r_1$, $r_2$.
Then we say that $T_{flag}=\mathcal{K} \ast f$ is a \textit{flag singular integral} $\mathbb R^n\times \mathbb R^m\times \mathbb R^d$.

\end{definition}

In \cite{HLW}, Han et al. obtained the boundedness of the flag singular integrals on 
weighted flag Hardy spaces. Inspired by this result, we prove that the flag singular integrals are also bounded on the weighted product Hardy spaces $H^p_{prod,w}(\mathbb R^n\times \mathbb R^m\times \mathbb R^d)$ as follows. 

\begin{theorem}\label{fp-3}
Suppose that $w\in A_{\infty}(\mathbb R^n\times \mathbb R^m\times \mathbb R^d)$. If $T_{flag}$ is a a flag singular integral on $\mathbb R^n\times \mathbb R^m\times \mathbb R^d$, then $T_{flag}$ is bounded on weighted product Hardy spaces $H^p_{prod,w}(\mathbb R^n\times \mathbb R^m\times \mathbb R^d)$ for $0<p\le1$.
\end{theorem}

This paper is organized as follows. In Section \ref{pre}, we give some known results on weighted flag and product Hardy spaces along with the Fefferman--Stein vector valued maximal inequality
on product domain. Then we prove the boundedness of one-parameter singular integrals in the two cases of $\mathbb R^n\times \mathbb R^m$ and $\mathbb R^n\times \mathbb R^m\times \mathbb R^d$ on the corresponding weighted flag Hardy spaces in Section 3. In Section 4, we focus on the boundedness of  
one-parameter singular integrals in the two cases of $\mathbb R^n\times \mathbb R^m$ and $\mathbb R^n\times \mathbb R^m\times \mathbb R^d$ on the corresponding weighted product Hardy spaces.
In Section 5, we show that the flag singular integrals on $\mathbb R^n\times \mathbb R^m$ and $\mathbb R^n\times \mathbb R^m\times \mathbb R^d$ are bounded on the corresponding weighted product Hardy spaces by using the weighted product Hardy spaces boundedness of the three-parameter singular integrals. In this last section, we will conclude by briefly considering the extension of the main results in this paper to Hardy spaces for ball quasi-Banach function spaces.

Throughout the paper, the letter
$C$ denotes a positive constant that may vary at each occurrence
but is independent to the main parameter, $A\lesssim B$ means that there are a constant
$C>0$ independent of the the main parameter such that $A\leq CB$ and $A\sim B$ denotes that $A\lesssim B$ and $B\lesssim A$.

\section{Preliminaries}\label{pre}
In this section, we recall some definitions and the basic results on weighted flag and product Hardy spaces. 
\subsection{Strong Hardy--Littlewood maximal operators on weighted Lebesgue spaces}\quad

For convenience, $\mathbb R^\mathcal X$ can represent either
$\mathbb R^n\times \mathbb R^m$ or $\mathbb R^n\times \mathbb R^m\times \mathbb R^d$.
Denote by $L_{{\rm loc}}^{1}\left(\mathbb{R}^\mathcal X \right)$ the set of all locally integrable functions on $\mathbb{R}^\mathcal X$. 
In the multi-parameter case, it is natural to replace the classical Hardy--Littlewood maximal operator $M$ by the {\it strong Hardy--Littlewood maximal operators} $M_s$, which is defined by setting, for all $f\in L_{{\rm loc}}^{1}\left(\mathbb{R}^\mathcal X\right)$ and $x\in \mathbb{R}^\mathcal X$,
\begin{align*}\label{eq:second7}
	M_sf\left(x\right):= \sup\limits_{x\in R} \frac{1}{|R|} \int_{R} \left|f\left(y\right)\right|\,dy,
\end{align*}
where the supremum of $R$ are taken over all rectangles $R\in \mathbb R^\mathcal X$.
First we recall the definitions of product weights in the multi-parameter setting. For $1<p<\infty$, a nonnegative locally integrable function $w \in A_p\left(\mathbb{R}^\mathcal X\right)$ if there exists a constant $C>0$ such that
$$
\left(\frac{1}{|R|} \int_R w(x) d x\right)\left(\frac{1}{|R|} \int_R w(x)^{-1 /(p-1)} d x\right)^{p-1} \leqslant C.
$$
We say $w \in A_1\left(\mathbb{R}^\mathcal X\right)$ if there exists a constant $C>0$ such that
$$
M_s w(x) \leqslant C w(x) 
$$
for almost every $x \in \mathbb{R}^\mathcal X$.
We define $w \in A_{\infty}\left(\mathbb{R}^\mathcal X\right)$ by
$$
A_{\infty}\left(\mathbb{R}^\mathcal X\right)=\bigcup_{1 \leqslant p<\infty} A_p\left(\mathbb{R}^\mathcal X\right)
$$
%Observe that if $w \in A_{\infty}\left(\mathbb{R}^\mathcal X\right)$, then $w \in A_{q_w}\left(\mathbb{R}^\mathcal X\right)$, where $q_w=\inf \left\{q: w \in A_q\left(\mathbb{R}^\mathcal X\right)\right\}$.

Next, we recall 
the well-known weighted boundedness of the strong maximal operators
$\mathcal M_s$ as follows.

\begin{proposition}\label{FS}\cite{gr85}\label{max}\quad Let $1<p<\infty$ and $w\in A_{p}(\mathbb R^\mathcal X).$ Then $\mathcal M_s$ is of type
$(L^p_w(\mathbb R^\mathcal X),L^p_w(\mathbb R^\mathcal X))$.
Moreover, for any $1<s<\infty$, 
\begin{align*}
		\left\|\left\{\sum_{j=1}^{\infty}\left[M_s\left(f_{j}\right)\right]^{s}\right\}^{1 / s}\right\|_{L^p_w(\mathbb R^\mathcal X)} 
		\le C\left\|\left\{\sum_{j=1}^{\infty}\left|f_{j}\right|^{s}\right\}^{1 / s}\right\|_{L^p_w(\mathbb R^\mathcal X)}
	\end{align*}
\end{proposition}

\subsection{Weighted product Hardy spaces}\quad

In the subsection, we recall the Littlewood--Paley characterization for the three-parameter weighted Hardy spaces on $\mathbb R^n\times \mathbb R^m\times \mathbb R^d$. For more details, we refer the readers
to \cite{Ru}. 
For conveniences,  hereafter $Q\in \mathcal R^{j,k,\ell}$ means that $Q=I\times J\times K\subset \mathbb R^n\times\mathbb R^m\times \mathbb R^d$
with side-length \(\ell(I) = 2^{j}, \ell(J) = 2^{k} \) and \( \ell(K) = 2^{\ell} \), \( x_Q \) denotes the left lower corner
of $Q$. Similarly, $Q\in \mathcal R^{j,k}$ means that $Q=I\times J\subset \mathbb R^n\times\mathbb R^m$
with side-length \(\ell(I) = 2^{j}\) and  \( \ell(J) = 2^{k} \).
The definition of weighted product Hardy spaces on $\mathbb R^n\times \mathbb R^m$ is nearly identical but easier, so we omit it.
Let \(\mathcal S(\mathbb R^n\times \mathbb R^m\times \mathbb R^d)\) denote the set of Schwartz functions in \(\mathbb R^n\times \mathbb R^m\times \mathbb R^d\)
and $x=(x_1,x_2,x_3)\in\mathbb R^n\times \mathbb R^m\times \mathbb R^d$. We denote \( f \in \mathcal{S}_0(\mathbb{R}^3)\),
if $f\in \mathcal S(\mathbb R^n\times \mathbb R^m\times \mathbb R^d)$ and
\text{for\, all} indices $\alpha, \beta, \gamma$ of nonnegative integers,
\[\int_{\mathbb{R}} f(x_1,x_2,x_3)x_1^{\alpha} dx_1 = \int_{\mathbb{R}} f(x_1,x_2,x_3)x_2^{\beta} dx_2 = \int_{\mathbb{R}} f(x_1,x_2,x_3)x_3^{\gamma} dx_3 = 0.\]
The semi-norms of $f\in \mathcal{S}_0(\mathbb R^n\times \mathbb R^m\times \mathbb R^d)$ are the Schwartz semi-norms.
The dual space of \( \mathcal{S}_0(\mathbb R^n\times \mathbb R^m\times \mathbb R^d) \) is denoted by \( \mathcal{S}_0'(\mathbb R^n\times \mathbb R^m\times \mathbb R^d) \). 

Let \( \psi^{(1)} \in \mathcal{S}(\mathbb{R}^n), \psi^{(2)} \in \mathcal{S}(\mathbb{R}^m), \psi^{(3)} \in \mathcal{S}(\mathbb{R}^d) \) and satisfy
\begin{equation}\label{1.1}
\sum_{j \in \mathbb{Z}} |\widehat{\psi^{(1)}}(2^{-j} \xi_1)|^2 = 1 \quad \text{for all } \xi_1 \in \mathbb{R}^n \setminus \{0\},
\end{equation}
\begin{equation}\label{1.2}
\sum_{k \in \mathbb{Z}} |\widehat{\psi^{(2)}}(2^{-k} \xi_2)|^2 = 1 \quad \text{for all } \xi_2 \in \mathbb{R}^m \setminus \{0\},
\end{equation}
\begin{equation}\label{1.3}
\sum_{\ell \in \mathbb{Z}} |\widehat{\psi^{(3)}}(2^{-\ell} \xi_3)|^2 = 1 \quad \text{for all } \xi_3 \in \mathbb{R}^d \setminus \{0\},
\end{equation}
and the support conditions
\begin{equation}\label{1.4}
\mbox{supp}\;\widehat{\psi^{(i)}}(\xi_i)\subset \{\xi_i: \frac{1}{2}< |\xi_i|\le 2\},
\end{equation}\
for $i=1,2,3$. Denote
\[
\psi_{j,k,\ell}(x_1,x_2,x_3) = 2^{-j-k-\ell} \psi^{(1)}(2^{-j} x_1) \psi^{(2)}(2^{-k}x_2)  \psi^{(3)}(2^{-\ell} x_3),
\]
where
\[
\psi_j^{(1)}(x_1) = 2^{-j} \psi^{(1)}(2^{-j} x_1), \quad \psi_k^{(2)}(x_2) = 2^{-k} \psi^{(2)}(2^{-k} x_2), \quad \psi_{\ell}^{(3)}(x_3) = 2^{-\ell} \psi^{(3)}(2^{-\ell} x_3).
\]

By the Fourier transform, we can obtain the continuous version of the Calder\'on identity on \( L^2(\mathbb R^n\times \mathbb R^m\times \mathbb R^d) \), that is, for any \( f \in L^2(\mathbb R^n\times \mathbb R^m\times \mathbb R^d) \)
and $u:=(x,y,z)\in \mathbb R^n\times \mathbb R^m\times \mathbb R^d$,
\[
f(x,y,z) = \sum_{j,k,\ell \in \mathbb{Z}} \psi_{j,k,\ell} \ast \psi_{j,k,\ell} \ast f(x,y,z). 
\]

We also need the following discrete Calder\'on identity in $\mathcal{S}_0'(\mathbb R^n\times \mathbb R^m\times \mathbb R^d)$, whose proof is nearly identical to that of Theorem 1.1 in \cite{HLLTW}.
Also see \cite{FJ,FHLL} for more details on the Calder\'on  reproducing formula.

\begin{proposition}\label{disc} Let \( \psi_{j,k,\ell} \) be the same as in (\ref{1.1})-(\ref{1.4}). Then
\begin{align}\label{dsc}
f(u) = \sum_{j,k,\ell} \sum_{Q\in \mathcal R^{j,k,\ell}} | Q | {\psi}_{j,k,\ell}(u-u_Q) \ast \psi_{j,k,\ell} \ast f (u_Q),
\end{align}
where the series in (\ref{dsc}) converges in the norm of \(L^2(\mathbb R^n\times\mathbb R^m\times \mathbb R^d)\), in the norm of \( \mathcal{S}_0(\mathbb R^n\times\mathbb R^m\times \mathbb R^d) \) and in the dual space \( \mathcal{S}_0'(\mathbb R^n\times\mathbb R^m\times \mathbb R^d) \).
\end{proposition}

For \( f \in \mathcal{S}_0'(\mathbb R^n\times\mathbb R^m\times \mathbb R^d) \) and $x\in \mathbb R^n\times\mathbb R^m\times \mathbb R^d$, we define the {\it discrete Littlewood--Paley square function} on \(\mathbb R^n\times\mathbb R^m\times \mathbb R^d\) by
\[
\mathcal{G}^d(f)(u) := \left(  \sum_{j,k,\ell} \sum_{Q\in \mathcal R^{j,k,\ell}} |\psi_{j,k,\ell} \ast f (u_Q) |^2 \chi_{Q}(u_Q) \right)^{\frac{1}{2}},
\]
where \( \psi_{j,k,\ell} \) satisfies the same conditions as (\ref{1.1})-(\ref{1.4}). 
When \( 1 < p < \infty \), by the iteration argument, vector-valued Littlewood--Paley--Stein estimates and the duality argument, we can get \[ \| f \|_{L^{p}(\mathbb R^n\times\mathbb R^m\times \mathbb R^d)} \sim \| \mathcal{G}^d(f) \|_{L^{p}(\mathbb R^n\times\mathbb R^m\times \mathbb R^d)}. \]

\begin{definition}
For $0<p<\infty$ and $w\in A_\infty(\mathbb R^n\times\mathbb R^m\times \mathbb R^d)$, we define the {\it weighted product Hardy space}
$$H^p_w(\mathbb R^n\times\mathbb R^m\times \mathbb R^d)
=\{f\in \mathcal{S}_0'(\mathbb R^n\times\mathbb R^m\times \mathbb R^d):\mathcal{G}_d(f)\in L^p_w (\mathbb R^n\times\mathbb R^m\times \mathbb R^d)\}$$
with the norm
$$\|f\|_{H^p_w(\mathbb R^n\times\mathbb R^m\times \mathbb R^d)}
=\|\mathcal{G}^d(f)\|_{L^p_w(\mathbb R^n\times\mathbb R^m\times \mathbb R^d)}.$$
\end{definition}

%Then by using the Min--Max type inequality together with the above discrete Calder\'on identity, one can get the discrete Littlewood–Paley–Stein characterization for the three-parameter weighted Hardy spaces. Therefore, the definition of these weighted Hardy spaces is independent of the choice of functions \( \psi_{j,k,\ell} \) and thus ensure that the weighted Hardy space is well defined.

%% Note that in the example below, the braces { } around \cite are necessary (due to nested optional parameters)

%\begin{theorem}[Identity Principle, see also {\cite[Theorem 5]{HillDow}}]\label{T:1}
%If $A=B$, then the following conditions are equivalent: 

%\begin{enumerate}[label=\upshape(\roman*), leftmargin=*, widest=iii]
%\item first item,\label{it:1}
%\item second item,\label{it:2}
%\item third item.\label{it:3}
%\end{enumerate}
%\end{theorem}

%\begin{equation}\label{E:1}
%\begin{aligned}[t]
%\quad + \Bigl(\prod_{i=1}^n A_i\Bigr)^2 + \biggl(\frac{u}{v}\biggr)^n\\
%&\overset{\alpha}= 
%\begin{cases}
%1 &\text{if $x \in (0, \pi)$,}\\
%0 &\text{otherwise.}
%\end{cases}
%\end{aligned}
%\end{equation}

\subsection{Weighted flag Hardy spaces}\quad

To recall the weighted flag Hardy space theory, 
appropriate test functions and distributions are needed.
For this purpose, we recall flag test functions as follows.
The definition of weighted flag Hardy spaces on $\mathbb R^n\times \mathbb R^m$ is very similar but easier, so for simplicity we only give the precise definitions of weighted flag Hardy spaces on $\mathbb R^n\times \mathbb R^m\times \mathbb R^d$. For more details
on flag singular integrals and flag Hardy spaces, see \cite{CHW, HLW,WW}. 

\begin{definition} A Schwartz function $f$ on $\mathbb{R}^N$ is said to be a {\it flag test function} in $\mathcal{S}_{\mathcal{F}}\left(\mathbb{R}^N\right)$ if it satisfies the partial cancellation conditions
$$\int_{\mathbb{R}^{d}} f\left(x, y, z\right) z^\alpha dz=0$$ 
for all multi-indices $\alpha$ and every $\left(x, y\right) \in \mathbb{R}^{n+m}$.
The semi-norms on $\mathcal{S}_{\mathcal{F}}\left(\mathbb{R}^N\right)$ are the same as the classical ones on $\mathcal{S}\left(\mathbb{R}^N\right)$ and these semi-norms make $\mathcal{S}_{\mathcal{F}}$ a topological vector space. Let $\mathcal{S}_{\mathcal{F}}^{\prime}\left(\mathbb{R}^N\right)$ denote the topological dual space of $\mathcal{S}_{\mathcal{F}}\left(\mathbb{R}^N\right)$.
\end{definition}

Let $N_1=n+m+d, N_2=m+d$ and $N_3=d$. For $i=1,2,3$, let $\psi^{(i)} \in \mathcal{S}\left(\mathbb{R}^{N_i}\right)$ satisfy
$$
\operatorname{supp} \widehat{\psi^{(i)}} \subset\left\{\xi^i \in \mathbb{R}^{N_i}: 1 / 2<\left|\xi^i\right| \leq 2\right\}
$$
and
$$
\sum_{j_i \in \mathbb{Z}} \widehat{\psi^{(i)}}\left(2^{j_i} \xi^i\right)^2=1 \quad \text { for all } \xi^i \in \mathbb{R}^{N_i} \backslash\{0\} .
$$

Define $\psi_{j_i}^{(i)}\left(x^i\right)=2^{-j_i N_i} \psi^{(i)}\left(2^{-j_i} x^i\right), x^i \in \mathbb{R}^{N_i} \text { and } \widetilde{\psi}_{j_i}^{(i)}=\delta_{\mathbb{R}^{N-N_i}} \otimes \psi_{j_i}^{(i)}, i=1,2,3.$
For $J=\left(j_1, j_2, j_3\right) \in \mathbb{Z}^3$, set $\psi_J=\widetilde{\psi}_{j_1}^{(1)} * \widetilde{\psi}_{j_2}^{(2)} * \widetilde{\psi}_{j_3}^{(3)}$. 
Let $u:=(x_1,x_2,x_3)\in \mathbb R^n\times\mathbb R^m \times \mathbb R^d$.
Then we get the Calder\'on reproducing formula
$$
f(u)=\sum_{J \in \mathbb{Z}^3} \sum_{R \in \mathcal{R}_{\mathcal{F}}^J}|R| \psi_J\left(u-u_R\right) \psi_J * f\left(u_R\right),
$$
where $u_R$ denotes the "lower left corner" of $R$
and the series converges in $L^2(\mathbb R^N),$ $\mathcal S_{\mathcal F}(\mathbb R^N)$ and $\mathcal S'_{\mathcal F}(\mathbb R^N)$ whenever $f$ is in the corresponding space.
Based on the above reproducing formula, the \textit{Littlewood--Paley--Stein square function} of \( f \in S'_{\mathcal{F}}(\mathbb{R}^N) \) is defined by
\[
g_{\mathcal{F}}(f)(u) = \left( \sum_{J \in \mathbb{Z}^3} \sum_{R \in \mathcal{R}_{\mathcal{F}}^J} |\psi_J\ast f(u_R)|^2 \chi_R(u) \right)^{\frac{1}{2}},
\]
where \( \chi_R \) is the indicator function of \( R \).

\begin{definition}
Let \( 0 < p < \infty \) and \( w \in A_{\infty}(\mathbb{R}^n\times\mathbb R^m\times \mathbb R^d) \). The \textit{weighted flag Hardy space} \( H_{flag,w}^p(\mathbb{R}^n\times\mathbb R^m\times \mathbb R^d) \) is defined by
\[
H_{flag,w}^p(\mathbb{R}^n\times\mathbb R^m\times \mathbb R^d) = \{ f \in S'_{\mathcal{F}}(\mathbb{R}^n\times\mathbb R^m\times \mathbb R^d) : g_{\mathcal{F}}(f) \in L_w^p(\mathbb{R}^n\times\mathbb R^m\times \mathbb R^d) \}
\]
with quasi-norm \( \|f\|_{H_{flag,w}^p(\mathbb{R}^n\times\mathbb R^m\times \mathbb R^d)} := \|g_{\mathcal{F}}(f)\|_{L_w^p(\mathbb{R}^n\times\mathbb R^m\times \mathbb R^d)} \).
\end{definition}

From \cite[Corollary 2.6]{HLW}, it is know that $L^2(\mathbb R^{n+m+d})\cap H_{flag,w}^p(\mathbb R^n\times \mathbb R^m\times \mathbb R^d)$ is dense in $H_{flag,w}^p(\mathbb R^n\times \mathbb R^m\times \mathbb R^d)$. 

\section{Proofs of Theorems \ref{1f} and \ref{1f-3}}

Before we give the proof of Theorem \ref{1f}, we establish the following flag type almost
orthogonal estimates for the one-parameter singular integral kernels.
\begin{lemma}\label{orth2}
For all $z=(x,y)\in \mathbb R^n\times \mathbb R^m$, suppose that $\phi^{(1)}\in \mathcal S(\mathbb R^{n+m})$ with support in $B(0,2)$
and $\int_{\mathbb R^{n+m}}\phi^{(1)}(x,y)dxdy=0$ and $\phi^{(2)}\in \mathcal S(\mathbb R^{m})$ and $\int_{\mathbb R^{m}}\phi^{(2)}(y)dy=0$.
Then
\begin{align*}%\label{3.1}
|\phi_{j,k}\ast K\ast \phi_{j',k'}(x,y)|\lesssim 2^{-|j-j'|-|k-k'|}\frac{2^{(j\vee j')/2}}{(2^{j\vee j'}+|x|)^{n+1/2}} \frac{2^{(j\vee j')/2}\vee 2^{(k\vee k')/2}}{(2^{j\vee j'}\vee 2^{k\vee k'}+|y|)^{m+1/2}}.
\end{align*}

\end{lemma}

\begin{proof}
Write $z=(x,y)\in\mathbb R^n\times\mathbb R^m$. First we observe that 
\begin{align}\label{3.2}
|K\ast \phi^{(1)}(z)|\lesssim \frac{1}{(1+|z|)^{n+m+1}},
\end{align}
and that
\begin{align}\label{3.3}
|\phi_{j}\ast \phi_{j'}(z)|\lesssim 2^{-|j-j'|}\frac{2^{j\vee j'}}{(2^{j\vee j'}+|z|)^{n+m+1}}.
\end{align}

By a dilation argument together with (\ref{3.2}) and (\ref{3.3}), we conclude that
\begin{align}\label{3.4}
|\phi_{j}\ast K\ast \phi_{j'}(z)|
=|K\ast \phi_{j}\ast \phi_{j'}(z)|
\lesssim 2^{-|j-j'|}\frac{2^{j\vee j'}}{(2^{j\vee j'}+|z|)^{n+m+1}}.
\end{align}

For any $y\in\mathbb R^m$ by applying the classical almost orthogonal estimate again we obtain that
\begin{align}\label{3.5}
|\phi_{k}\ast \phi_{k'}(y)|\lesssim 2^{-|k-k'|}
\frac{2^{k\vee k'}}{(2^{k\vee k'}+|y|)^{m+1}}.
\end{align}

%Note that 
%\begin{align*}
%\phi_{j,k}\ast K\ast \phi_{j',k'}(x,y)&=(\phi_{j}\ast_2\phi_{k})\ast K
%\ast_2 (\phi_{j'}\ast \phi_{k'})(x,y)\\
%&=(\phi_{j}\ast K\ast \phi_{j'})\ast_2 (\phi_{k}\ast \phi_{k'})(x,y).
%\end{align*}

Then by (\ref{3.4}) and (\ref{3.5}) we have
\begin{align*}
&\phi_{j,k}\ast K\ast \phi_{j',k'}(x,y)\\
&=(\phi_{j}\ast_2\phi_{k})\ast K
\ast_2 (\phi_{j'}\ast \phi_{k'})(x,y)\\
&=(\phi_{j}\ast K\ast \phi_{j'})\ast_2 (\phi_{k}\ast \phi_{k'})(x,y)\\
&\lesssim 2^{-|j-j'|}2^{-|k-k'|}\int_{\mathbb R^m}
\frac{2^{j\vee j'}}{(2^{j\vee j'}+(|x|^2+|y-z|^2)^{1/2})^{n+m+1}}
\frac{2^{k\vee k'}}{(2^{k\vee k'}+|z|)^{m+1}}dz\\
&\lesssim 2^{-|j-j'|}2^{-|k-k'|}\int_{\mathbb R^m}
\frac{2^{j\vee j'}}{(2^{j\vee j'}+(|x|+|y-z|))^{n+m+1}}
\frac{2^{k\vee k'}}{(2^{k\vee k'}+|z|)^{m+1}}dz.
\end{align*}

{\bf Case 1:} When $j\vee j' \le k\vee k'$ and $|y|\ge 2^{k\vee k'}$,
write 
\begin{align*}
&\int_{\mathbb R^m}
\frac{2^{j\vee j'}}{(2^{j\vee j'}+(|x|+|y-z|))^{n+m+1}}
\frac{2^{k\vee k'}}{(2^{k\vee k'}+|z|)^{m+1}}dz\\
=&\int_{|z|\le\frac{1}{2}\,\, {or}\,\, |z|\ge 2|y|}+\int_{\frac{1}{2}|y|\le |z|\le 2|y|}
\cdots=:I+II.
\end{align*}
For term $I$, we find that
\begin{align*}
I&\lesssim
\frac{2^{j\vee j'}}{(2^{j\vee j'}+(|x|+|y|))^{n+m+1}}\\
&\lesssim
\frac{2^{(j\vee j')/2}}{(2^{j\vee j'}+|x|)^{n+1/2}}
\frac{2^{(j\vee j')/2}}{|y|^{m+1/2}}\\
&\lesssim
\frac{2^{(j\vee j')/2}}{(2^{j\vee j'}+|x|)^{n+1/2}}
\frac{2^{(k\vee k')/2}}{(2^{k\vee k'}+|y|)^{m+1/2}}.
\end{align*}

Next we estimate the term $II$.
\begin{align*}
II&\lesssim \int_{\frac{1}{2}|y|\le |z|\le 2|y|}
\frac{2^{j\vee j'}}{(2^{j\vee j'}+(|x|+|y-z|))^{n+m+1}}
\frac{2^{k\vee k'}}{(2^{k\vee k'}+|y|)^{m+1}}dz\\
&\lesssim \frac{2^{k\vee k'}}{(2^{k\vee k'}+|y|)^{m+1}}
\int_{\mathbb R^m}
\frac{2^{j\vee j'}}{(2^{j\vee j'}+(|x|+|y-z|))^{n+m+1}}
dz\\
&\lesssim \frac{2^{k\vee k'}}{(2^{k\vee k'}+|y|)^{m+1}}
\frac{2^{j\vee j'}}{(2^{j\vee j'}+|x|)^{n+1}}\\
&\lesssim \frac{2^{(j\vee j')/2}}{(2^{j\vee j'}+|x|)^{n+1/2}}
\frac{2^{(k\vee k')/2}}{(2^{k\vee k'}+|y|)^{m+1/2}}.
\end{align*}

{\bf Case 2:} When $j\vee j' \le k\vee k'$ and $|y|\le 2^{k\vee k'}$,
then 
\begin{align*}
&\int_{\mathbb R^m}
\frac{2^{j\vee j'}}{(2^{j\vee j'}+(|x|+|y-z|))^{n+m+1}}
\frac{2^{k\vee k'}}{(2^{k\vee k'}+|z|)^{m+1}}dz\\
&\le\frac{1}{(2^{k\vee k'})^{m}}\int_{\mathbb R^m}
\frac{2^{j\vee j'}}{(2^{j\vee j'}+(|x|+|y-z|))^{n+m+1}}
dz\\
&\lesssim \frac{2^{k\vee k'}}{(2^{k\vee k'}+|y|)^{m+1}}
\frac{2^{j\vee j'}}{(2^{j\vee j'}+|x|)^{n+1}}\\
&\lesssim \frac{2^{(j\vee j')/2}}{(2^{j\vee j'}+|x|)^{n+1/2}}
\frac{2^{(k\vee k')/2}}{(2^{k\vee k'}+|y|)^{m+1/2}}.
\end{align*}

{\bf Case 3:} When $j\vee j' \ge k\vee k'$ and $|y|\le 2^{j\vee j'}$,
then 
\begin{align*}
&\int_{\mathbb R^m}
\frac{2^{j\vee j'}}{(2^{j\vee j'}+(|x|+|y-z|))^{n+m+1}}
\frac{2^{k\vee k'}}{(2^{k\vee k'}+|z|)^{m+1}}dz\\
&\le
\frac{2^{j\vee j'}}{(2^{j\vee j'}+|x|)^{n+m+1}}\\
&\lesssim \frac{2^{(j\vee j')/2}}{(2^{j\vee j'}+|x|)^{n+1/2}}
\frac{2^{(j\vee j')/2}}{(2^{j\vee j'}+|y|)^{m+1/2}}.
\end{align*}

{\bf Case 4:} When $j\vee j' \ge k\vee k'$ and $|y|\ge 2^{j\vee j'}$,
repeating similar argument in {Case 1}, we write 
\begin{align*}
&\int_{\mathbb R^m}
\frac{2^{j\vee j'}}{(2^{j\vee j'}+(|x|+|y-z|))^{n+m+1}}
\frac{2^{k\vee k'}}{(2^{k\vee k'}+|z|)^{m+1}}dz\\
=&\int_{|z|\le\frac{1}{2}\,\, {or}\,\, |z|\ge 2|y|}+\int_{\frac{1}{2}|y|\le |z|\le 2|y|}
\cdots=:I+II.
\end{align*}
For term $I$, we find that
\begin{align*}
I&\lesssim
\frac{2^{j\vee j'}}{(2^{j\vee j'}+(|x|+|y|))^{n+m+1}}\\
&\lesssim \frac{2^{(j\vee j')/2}}{(2^{j\vee j'}+|x|)^{n+1/2}}
\frac{2^{(j\vee j')/2}}{(2^{j\vee j'}+|y|)^{m+1/2}}.
\end{align*}

Next we estimate the term $II$.
\begin{align*}
II&\lesssim \int_{\frac{1}{2}|y|\le |z|\le 2|y|}
\frac{2^{j\vee j'}}{(2^{j\vee j'}+(|x|+|y-z|))^{n+m+1}}
\frac{2^{k\vee k'}}{(2^{k\vee k'}+|y|)^{m+1}}dz\\
&\lesssim \frac{2^{k\vee k'}}{(2^{k\vee k'}+|y|)^{m+1}}
\int_{\mathbb R^m}
\frac{2^{j\vee j'}}{(2^{j\vee j'}+(|x|+|y-z|))^{n+m+1}}
dz\\
&\lesssim \frac{2^{j\vee j'}}{(2^{j\vee j'}+|y|)^{m+1}}
\frac{2^{j\vee j'}}{(2^{j\vee j'}+|x|)^{n+1}}\\
&\lesssim \frac{2^{(j\vee j')/2}}{(2^{j\vee j'}+|x|)^{n+1/2}}
\frac{2^{(j\vee j')/2}}{(2^{j\vee j'}+|y|)^{m+1/2}}.
\end{align*}
Therefore, we have completed the proof of this lemma.
\end{proof}

To prove Theorem \ref{1f}, we also need the following discrete Calder\'on reproducing formula on \( L^2(\mathbb{R}^{n+m}) \). Precisely, set \( \phi^{(1)} \in C_0^\infty(\mathbb{R}^{n+m}) \) with
\[
\int_{\mathbb{R}^{n+m}} \phi^{(1)}(x, y) x^\alpha y^\beta \, dx \, dy = 0, \quad \text{for all } \alpha, \beta \text{ satisfying } 0 \leq |\alpha|,\,|\beta| \leq M_0,
\]
where \( M_0 \) is a large positive integer, and
\[
\sum_j |\widehat{\phi^{(1)}}(2^{-j} \xi_1, 2^{-j} \xi_2)|^2 = 1, \quad \text{for all } (\xi_1, \xi_2) \in \mathbb{R}^{n+m} \setminus \{(0, 0)\},
\]
and take \( \phi^{(2)} \in C_0^\infty(\mathbb{R}^m) \) with
\[
\int_{\mathbb{R}^m} \phi^{(2)}(z) z^\gamma \, dz = 0 \quad \text{for all } 0 \leq |\gamma| \leq M_0,
\]
and
\[
\sum_k |\widehat{\phi^{(2)}}(2^{-k} \xi_2)|^2 = 1 \quad \text{for all } \xi_2 \in \mathbb{R}^m \setminus \{0\}.
\]

Furthermore, we may assume that \( \phi^{(1)} \) and \( \phi^{(2)} \) are radial functions and supported in the unit balls of \( \mathbb{R}^{n+m} \) and \( \mathbb{R}^m \) respectively. Set again
\[
\phi_{j k}(x, y) = \int_{\mathbb{R}^m} \phi_j^{(1)}(x, y - z) \phi_k^{(2)}(z)dz.
\]

For our purpose, we recall the discrete type reproducing formula, whose proof are nearly identical to \cite[Theorem 4.4]{HL} and \cite[Theorem 3.3]{HLW}.

\begin{theorem}\label{dis}
There exist functions \( \tilde{\phi}_{j k} \) and an operator \(T_N^{-1}\) such that
\[
f(x, y) = \sum_{j,k}\sum_{I,J}|I||J|\tilde{\phi}_{j,k}(x-x_I, y-y_J) \phi_{j, k}\ast\left(T_N^{-1}(f) \right)(x_I, y_J),
\]
where 
$I\times J\in \mathcal R^{j-N,k-N}$ and
functions $\widetilde{\phi}_{j,k}\left(x-x_I, y-y_J\right)$ satisfy the same conditions as ${\phi_{j,k}}$ with $\alpha_1,$ $\beta_1,$ $\gamma_1, N, M$ depending on $M_0, x_0=x_I$ and $y_0=y_J$. Moreover, $T_N^{-1}$ is bounded on $L^2\left(\mathbb R^{n+m}\right)$ and $H_{flag,w}^p\left(\mathbb R^n \times \mathbb R^m\right)$, and the series converges in $L^2\left(\mathbb R^{n+m}\right)$.
\end{theorem}

The following lemma in \cite[Lemma 2.3]{CHW} also plays a key role in the proof of Theorem \ref{1f}.
%whose proof can be found in \cite[Lemma 3.7]{HL}. 
\begin{lemma}\label{max}
Let \( I, I', J, J' \) be dyadic cubes in \( \mathbb{R}^n \) and \( \mathbb{R}^m \) respectively such that
$
\ell(I) = 2^{j}, \ell(J) = 2^{j\vee k}, \ell(I') = 2^{j'}\, \text{and}\; \ell(J') = 2^{j'\vee k'}.
$
Thus for any \( u, u' \in I\) and \( v, v' \in J \), we have
\begin{align*}
&\sum_{I', J'} \frac{2^{(j\vee j' + j\vee j' \vee k\vee k')/2} |I'| |J'| |\phi_{j', k'} \ast f(x_{I'}, y_{J'})|}{(2^{j\vee j'} + |u' - x_{I'}|)^{n+1/2} (2^{j\vee j' \vee k\vee k'} + |v' - y_{J'}|)^{m+1/2}}\\
 &\leq C_1\left\{ M_s \left( \sum_{J'} \sum_{I'} |\phi_{j', k'}\ast f(x_{I'}, y_{J'})| \chi_{J'} \chi_{I'} \right)^r \right\}^{\frac{1}{r}} (u, v)
\end{align*}
where \( M_s \) is the strong maximal function on \( \mathbb{R}^n \times \mathbb{R}^m \), \(  \frac{2n}{2n+1}\vee \frac{2m}{2m+1}  < r\le 1 \) and $C_1=2^{[n(j\vee j'-j')+m(j\vee j'\vee k\vee k'-j'\vee k')](\frac{1}{r}-1)}$.
\end{lemma}
\noindent{\bf Proof of Theorem \ref{1f}}\quad
By using the fact that $L^2(\mathbb R^{n+m})\cap H_{flag,w}^p(\mathbb R^n\times \mathbb R^m)$ is dense in $H_{flag,w}^p(\mathbb R^n\times \mathbb R^m)$, to show this theorem we only need to prove that for every
$f\in L^2(\mathbb R^{n+m})\cap H_{flag,w}^p(\mathbb R^n\times \mathbb R^m)$,
$$
\left\| \left( \sum_{j,k} \sum_{I,J} |\phi_{j,k}\ast T(f)(x_I,y_J)|^2 \chi_I\chi_J \right)^{\frac{1}{2}}\right\|_{L_w^p(\mathbb R^n\times\mathbb R^m)}\lesssim \|f\|_{ H_{flag,w}^p(\mathbb R^n\times \mathbb R^m)}.
$$

To prove it, applying Theorem \ref{dis} for $f$ in $\phi_{j,k}\ast T(f)(x_I,y_J)$ we conclude that
\begin{align*}
&\phi_{j,k}\ast Tf(x_I, y_J)\\
&=\phi_{j,k}\ast K\ast\left(\sum_{j',k'}\sum_{I',J'}|I'||J'|\tilde{\phi}_{j', k'}(x_I-x_{I'}, y_J-y_{J'}) \phi_{j', k'}\ast\left(T_N^{-1}(f) \right)(x_{I'}, y_{J'})\right)\\
&=\sum_{j',k'}\sum_{I',J'}|I'||J'|\phi_{j,k}\ast K\ast\tilde{\phi}_{j', k'}(x_I-x_{I'}, y_J-y_{J'}) \phi_{j', k'}\ast\left(T_N^{-1}(f) \right)(x_{I'}, y_{J'}).
\end{align*}

Then by Lemma \ref{orth2} and Lemma \ref{max}, for any $u\in I$ and $v\in J$ we have
\begin{align*}
&\phi_{j,k}\ast Tf(x_I, y_J)\\
\lesssim& \sum_{j',k'}2^{-|j-j'|-|k-k'|}\sum_{I',J'}|I'||J'|\frac{2^{(j\vee j')/2}}{(2^{j\vee j'}+|x_I-x_{I'}|)^{n+1/2}}\\&\times \frac{2^{(j\vee j')/2}\vee 2^{(k\vee k')/2}}{(2^{j\vee j'}\vee 2^{k\vee k'}+|y_J-y_{J'}|)^{m+1/2}}
\phi_{j', k'}\ast\left(T_N^{-1}(f) \right)(x_{I'}, y_{J'})\\
\lesssim& \sum_{j',k'}2^{-|j-j'|-|k-k'|}
C_1\left\{ M_s \left( \sum_{J'} \sum_{I'} |\phi_{j', k'}\ast T_N^{-1}(f)(x_{I'}, y_{J'})| \chi_{J'} \chi_{I'} \right)^r \right\}^{\frac{1}{r}} (u, v),
\end{align*}
where $C_1=2^{[n(j\vee j'-j')+m(j\vee j'\vee k\vee k'-j'\vee k')](\frac{1}{r}-1)}$ and
$\frac{2n}{2n+1}\vee\frac{2m}{2m+1}<r\le 1$.

Then with the help of the H\"older inequality, we deduce that
\begin{align*}
&\left\| \left( \sum_{j,k} \sum_{I,J} |\phi_{j,k}\ast T(f)(x_I,y_J)|^2 \chi_I\chi_J \right)^{\frac{1}{2}}\right\|_{L_w^p(\mathbb R^n\times\mathbb R^m)}\\
&\lesssim \Bigg\| \Bigg( \sum_{j,k} \sum_{I,J} \Bigg|\sum_{j',k'}2^{-|j-j'|-|k-k'|}
2^{[n(j\vee j'-j')+m(j\vee j'\vee k\vee k'-j'\vee k')](\frac{1}{r}-1)}\\
&\Bigg\{ M_s \left( \sum_{J'} \sum_{I'} |\phi_{j', k'}\ast T_N^{-1}(f)(x_{I'}, y_{J'})| \chi_{J'} \chi_{I'} \right)^r \Bigg\}^{\frac{1}{r}}\Bigg|^2 \chi_I\chi_J \Bigg)^{\frac{1}{2}}\Bigg\|_{L_w^p(\mathbb R^n\times\mathbb R^m)}\\
&\lesssim \Bigg\| \Bigg( \sum_{j,k} \sum_{I,J} \sum_{j',k'}2^{-|j-j'|-|k-k'|}
2^{[n(j\vee j'-j')+m(j\vee j'\vee k\vee k'-j'\vee k')](\frac{1}{r}-1)}\\
&\Bigg\{ M_s \left( \sum_{J'} \sum_{I'} |\phi_{j', k'}\ast T_N^{-1}(f)(x_{I'}, y_{J'})| \chi_{J'} \chi_{I'} \right)^r \Bigg\}^{\frac{2}{r}}\chi_I\chi_J \Bigg)^{\frac{1}{2}}\Bigg\|_{L_w^p(\mathbb R^n\times\mathbb R^m)}.
\end{align*}

Denote $\tau=\frac{2n}{2n+1}\vee\frac{2m}{2m+1}.$
 If $w\in A_{p/\tau}(\mathbb R^n\times \mathbb R^m)$, then there exist $\tau<r\le1$ such that
 $w\in A_{p/r}(\mathbb R^n\times \mathbb R^m)$.
Hence,
summing over $Q=I\times J$
%using the fact that $$\sum_{j',k'}2^{-|j-j'|-|k-k'|}
%2^{[n(j\vee j'-j')+m(j\vee j'\vee k\vee k'-j'\vee k')](\frac{1}{r}-1)}<\infty,$$ 
and applying the Fefferman--Stein vector valued maximal inequality
in Proposition \ref{FS} yield that
\begin{align*}
&\left\| \left( \sum_{j,k} \sum_{I,J} |\phi_{j,k}\ast T(f)(x_I,y_J)|^2 \chi_I\chi_J \right)^{\frac{1}{2}}\right\|_{L_w^p(\mathbb R^n\times\mathbb R^m)}\\
&\lesssim \Bigg\| \Bigg( \sum_{j',k'}
\Bigg\{ M_s \left( \sum_{J'} \sum_{I'} |\phi_{j', k'}\ast T_N^{-1}(f)(x_{I'}, y_{J'})| \chi_{J'} \chi_{I'} \right)^r \Bigg\}^{\frac{2}{r}}\Bigg)^{\frac{1}{2}}\Bigg\|_{L_w^p(\mathbb R^n\times\mathbb R^m)}\\
&\lesssim \Bigg\| \Bigg( \sum_{j',k'}
\sum_{I',J'} |\phi_{j', k'}\ast T_N^{-1}(f)(x_{I'}, y_{J'})|^2 \chi_{I',J'} \chi_{I'}\Bigg)^{\frac{1}{2}}\Bigg\|_{L_w^p(\mathbb R^n\times\mathbb R^m)}\\
&\lesssim \|T_N^{-1}(f)\|_{ H_{flag,w}^p(\mathbb R^n\times \mathbb R^m)}\\
&\lesssim \|f\|_{ H_{flag,w}^p(\mathbb R^n\times \mathbb R^m)},
\end{align*}
where the last inequality follows from $T_N^{-1}$ is bounded on $H_{flag,w}^p\left(\mathbb R^n \times \mathbb R^m\right)$.
Thus, we have completed the proof of Theorem \ref{1f}.
\hfill $\square$

\noindent{\bf Proof of Theorem \ref{1f-3}}\quad
The proof is essentially identical to the proof of Theorem \ref{1f}. First we need the following almost orthogonal estimates.
Let $N_1=n+m+d, N_2=m+d$ and $N_3=d$. For $i=1,2,3$, let $\phi^{(i)} \in \mathcal{S}\left(\mathbb{R}^{N_i}\right)$ with support in $B(0,2)$
and $\int_{\mathbb R^{N_i}}\phi^{(i)}(x^i)dx^i=0$,
where $x^1=(x_1,x_2,x_3)$, $x^2=(x_2,x_3)$ and $x^3=x_3$.
For $J=\left(j_1, j_2, j_3\right) \in \mathbb{Z}^3$ and $J'=\left(j'_1, j'_2, j'_3\right) \in \mathbb{Z}^3$
and for all $u=(x_1,x_2,x_3)\in \mathbb R^n\times \mathbb R^m\times R^d$, then 
repeating the nearly same as the proof of Lemma \ref{orth2}
yields that
\begin{align*}%\label{3.1}
&|\phi_{J}\ast K\ast \phi_{J'}(u)|\lesssim 2^{-|j_1-j'_1|-|j_2-j'_2|-|j_3-j'_3|}\frac{2^{(j_1\vee j_1')/3}}{(2^{j_1\vee j'_1}+|x_1|)^{n+1/3}}\\
&\qquad\qquad\times\frac{2^{(j_1\vee j'_1)/3}\vee 2^{(j_2\vee j'_2)/3}}{(2^{j_1\vee j'_1}\vee 2^{j_2\vee j'_2}+|x_2|)^{m+1/3}}
\frac{2^{(j_1\vee j'_1)/3}\vee 2^{(j_2\vee j'_2)/3}\vee 2^{(j_3\vee j'_3)/3}}{(2^{j_1\vee j'_1}\vee 2^{j_2\vee j'_2}\vee 2^{(j_3\vee j'_3)}+|x_3|)^{d+1/3}}.
\end{align*}
 The rest of the proof is identical by applying the discrete type reproducing formula, H\"older's inequality
 along with the Fefferman–Stein vector valued maximal inequality.
\hfill $\square$

\section{Proofs of Theorems \ref{1p} and \ref{1p-3}}
The proof of Theorem \ref{1p} is similar but easier, so we only need to prove Theorem \ref{1p-3}.
To prove it, we need to define $\phi^{(1)}$, $\phi^{(2)}$ and $\phi^{(3)}$ to be the same as in (\ref{1.1})-(\ref{1.3}) with additional conditions that they are Schwartz functions supported in the unit balls centered at the origin in $\mathbb R^{n+m+d}$
and fulfilling the moment conditions
\begin{equation*}
\int_{\mathbb{R}^n} \phi^{(1)}(x)x^{\alpha} dx = \int_{\mathbb{R}^m} \phi^{(2)}(y)y^{\beta} dy = \int_{\mathbb{R}^d} \phi^{(3)}(z)z^{\gamma} dz = 0
\end{equation*}
for all multi-indices \( |\alpha|, |\beta|, |\gamma| \le M\), where $M$ is a fixed large positive integer. Denote
\[
\phi_{j,k,\ell}(x,y,z) = 2^{-jn-km-\ell d} \phi^{(1)}(2^{-j} x) \phi^{(2)}(2^{-k}y)  \phi^{(3)}(2^{-\ell}z),
\]
where
\[
\phi_j^{(1)}(x) = 2^{-jn} \phi^{(1)}(2^{-j} x), \quad \phi_k^{(2)}(y) = 2^{-km} \phi^{(2)}(2^{-k} y), \quad \phi_{\ell}^{(3)}(z) = 2^{-\ell d} \phi^{(3)}(2^{-\ell} z).
\]
First we obtain the following three-parameter almost orthogonal
estimates for the one-parameter singular integral kernels.

\begin{lemma}\label{ortho3} 
Suppose that $\phi$ is defined as above and $K$ is a one-parameter singular integral kernel in Definition \ref{1-3}. Then
\begin{align*}
&\left|\phi_{j,k,\ell} \ast {K} \ast \phi_{j',k',\ell'} (x, y, z) \right|\\
&\lesssim  2^{-|j-j'|}2^{-|k-k'|}2^{-|\ell-\ell'|}
\frac{2^{-(j \vee j')n}}{(1 + 2^{-(j \vee j')} |x|)^{n + 1}}
\frac{2^{-(k \vee k')m}}{(1 + 2^{-(k \vee k')} |y|)^{m + 1}}
\frac{2^{-(\ell \vee \ell')d}}{(1 + 2^{-(\ell \vee \ell')} |z|)^{d + 1}},
\end{align*}
for all $(x,y,z)\in \mathbb R^n\times\mathbb R^m\times \mathbb R^d$.
\end{lemma}

\begin{proof}
First we prove that
\begin{align*}
\left|{K} \ast \phi_{j,k,\ell} (x, y, z) \right|
\lesssim 
\frac{2^{-jn}}{(1 + 2^{-j} |x|)^{n + 1}}
\frac{2^{-km}}{(1 + 2^{-k} |y|)^{m + 1}}
\frac{2^{-\ell d}}{(1 + 2^{-\ell} |z|)^{d + 1}}.
\end{align*}
To end it, we prove it in eight cases as follows.

{\bf Case 1:} $|x| \geq 3\cdot2^{j}, |y| \geq 3\cdot2^{k}, |z| \geq 3\cdot2^{\ell}$. By the cancellation conditions of $\phi^{(i)}$, we rewrite
\begin{align*}
&{K}\ast\phi_{j,k,\ell}(x, y, z)\\
=& 2^{-jn-km-\ell d}\iiint \big[{K}(x-x_1, y-y_1, z-z_1)\\
&-{K}(x, y-y_1, z-z_1)-{K}(x-x_1, y, z-z_1)-{K}(x-x_1, y-y_1, z)\\
&+{K}(x, y, z-z_1)+{K}(x, y-y_1, z)+{K}(x-x_1, y, z)\\
&-{K}(x, y, z)\big] \phi^{(1)}(2^{-j}u_1) \phi^{(2)}(2^{-k}u_2)
\phi^{(3)}(2^{-\ell}u_3)dx_1dy_1 dz_1\\
=&2^{-jn-km-\ell d}\iiint \left\{\int_0^1\int_0^1\int_0^1 \partial^{1}_s\partial^{1}_t \partial^{1}_u 
K(x-sx_1,y-ty_1,z-uz_1)dsdtdu\right\}\\
=&-2^{-jn-km-\ell d}\iiint \bigg\{\int_0^1\int_0^1\int_0^1 \sum_i^n\sum_j^m\sum_\ell^d x_{1i}y_{1j}z_{1d} \partial^{1}_{xi}
\partial^{1}_{yj}\partial^{1}_{z\ell}\\
&\times K(x-sx_1,y-ty_1,z-uz_1)dsdtdu\bigg\}
\phi^{(1)}(2^{-j}x_1) \phi^{(2)}(2^{-k}y_1)
\phi^{(3)}(2^{-\ell}z_1)dx_1dy_1 dz_1.
\end{align*}
Then applying the regularity condition of $K$ yields that
\begin{align*}
&|{K}\ast\phi_{j,k,\ell}(x, y, z)| \\
&\lesssim 2^{-jn-km-\ell d}\int_{|x_1| \leq 2^{j}} \int_{|y_1| \leq 2^{k}} \int_{|z_1| \leq 2^{\ell}} 
\frac{|x_1| |y_1| |z_1|}{(|x|+|y|+|z|)^{n+m+d+3}}\\
&\quad\times
\phi^{(1)}(2^{-j}x_1) \phi^{(2)}(2^{-k}y_1)
\phi^{(3)}(2^{-\ell}z_1)dx_1dy_1 dz_1\\
&\lesssim \frac{2^{-jn-km-\ell d}}{|2^{-j}x|^{n+1} |2^{-k}y|^{m+1} 
|2^{-\ell}z|^{d+1}} \\
&\lesssim\frac{2^{-jn}}{(1 + 2^{-j} |x|)^{n + 1}}
\frac{2^{-km}}{(1 + 2^{-k} |y|)^{m + 1}}
\frac{2^{-\ell d}}{(1 + 2^{-\ell} |z|)^{d + 1}}.
\end{align*}

{\bf Case 2:} $|x| \geq 3\cdot2^{j}, |y| \geq 3\cdot2^{k}, |z| < 3\cdot2^{\ell}$. 
%From this and the support condition of $\phi_\ell^{(3)}(x_3-u_3)$, $|u_3|\le|u_3-x_3|+|x_3|\le4\cdot 2^{\ell}$. 2^{-j-k-\ell}\iiint \Delta_{1,-u_1}\Delta_{2,-u_2}K(x_1,x_2,x_3-u_3)\phi^{(1)}(2^{-j}u_1) \phi^{(2)}(2^{-k}u_2)
By the cancellation condition of $\phi^{(1)}$ and $\phi^{(2)}$, we obtain that
\begin{align*}
&{K}\ast\phi_{j,k,\ell}(x, y, z)\\
=& \iiint \big[{K}(x-x_1, y-y_1, z-z_1)
-{K}(x, y-y_1, z-z_1)\\
&-{K}(x-x_1, y, z-z_1)
+{K}(x, y, z-z_1)\big] \phi_j^{(1)}(x_1) \phi_k^{(2)}(y_1)
\phi_\ell^{(3)}(z_1)dx_1dy_1dz_1.
%=&\iiint \Delta_{1,-u_1}\Delta_{2,-u_2}K(x_1,x_2,x_3-u_3)\phi_j^{(1)}(u_1) \phi_k^{(2)}(u_2) \phi_\ell^{(3)}(u_3)du_3du_2 du_1
\end{align*}

Then repeating the similar argument to {\bf Case 1}, we get that
\begin{align*}
&|{K}\ast\phi_{j,k,\ell}(x, y, z)| \\
&\lesssim 2^{-jn-km-\ell d}\int_{|x_1| \leq 2^{j}} \int_{|y_1| \leq 2^{k}} \int_{|z_1| \leq 2^{\ell}} 
\frac{|x_1| |y_1|}{(|x|+|y|+|z-z_1|)^{n+m+d+2}}\\
&\quad\times
\phi^{(1)}(2^{-j}x_1) \phi^{(2)}(2^{-k}y_1)
\phi^{(3)}(2^{-\ell}z_1)dx_1dy_1 dz_1\\
&\lesssim \frac{2^{-jn-km-\ell d}}{|2^{-j}x|^{n+1} |2^{-k}y|^{m+1} 
} \\
&\lesssim\frac{2^{-jn}}{(1 + 2^{-j} |x|)^{n + 1}}
\frac{2^{-km}}{(1 + 2^{-k} |y|)^{m + 1}}
\frac{2^{-\ell d}}{(1 + 2^{-\ell} |z|)^{d + 1}},
\end{align*}
where the second inequality follows from the fact that for any $a$ is constant and $k>0$,
$$
\int_{\mathbb R^d} \frac{1}{(|a|+|z|)^{d+k}}dz\lesssim \frac{1}{|a|^k}.
$$

Notice that {\bf Case 3:} $|x| \geq 3\cdot2^{j}, |y|< 3\cdot2^{k}, |z| \geq 3\cdot2^{\ell}$
and
{\bf Case 4:} $|x| < 3\cdot2^{j}, |y| \geq 3\cdot2^{k}, |z| \geq 3\cdot2^{\ell}$ is nearly identical to
{\bf Case 2}, so we omit the proofs.

For {\bf Case 5}, $|x| \geq 3\cdot2^{j}, |y| <3\cdot2^{k}, |z| < 3\cdot2^{\ell}$. 
By applying the cancellation condition of $\phi^{(1)}$, we conclude that
\begin{align*}
&{K}\ast\phi_{j,k,\ell}(x, y, z)\\
=& \iiint \big[{K}(x-x_1, y-y_1, z-z_1)
-{K}(x, y-y_1, z-z_1)\big]\\
&\times\phi_j^{(1)}(x_1) \phi_k^{(2)}(y_1)
\phi_\ell^{(3)}(z_1)dx_1dy_1dz_1.
%=&\iiint \Delta_{1,-u_1}\Delta_{2,-u_2}K(x_1,x_2,x_3-u_3)\phi_j^{(1)}(u_1) \phi_k^{(2)}(u_2) \phi_\ell^{(3)}(u_3)du_3du_2 du_1
\end{align*}
By using the mean value theorem and the regularity condition of $K$, we get that
\begin{align*}
&|{K}\ast\phi_{j,k,\ell}(x, y, z)| \\
&\lesssim 2^{-jn-km-\ell d}\int_{|x_1| \leq 2^{j}} \int_{|y_1| \leq 2^{k}} \int_{|z_1| \leq 2^{\ell}} 
\frac{|x_1|}{(|x|+|y-y_1|+|z-z_1|)^{n+m+d+1}}\\
&\quad\times
\phi^{(1)}(2^{-j}x_1) \phi^{(2)}(2^{-k}y_1)
\phi^{(3)}(2^{-\ell}z_1)dx_1dy_1 dz_1.
\end{align*}

Noting that
\begin{align*}
&\int_{\mathbb R^m} \int_{\mathbb R^d} 
\frac{1}{(|x|+|y-y_1|+|z-z_1|)^{n+m+d+1}}dy_1 dz_1\\
&\lesssim \int_{\mathbb R^m}
\frac{1}{(|x|+|y-y_1|)^{n+m+1}}dy_1\\
&\lesssim 
\frac{1}{|x|^{n+1}}.
\end{align*}

Thus, we conclude that
\begin{align*}
&|{K}\ast\phi_{j,k,\ell}(x, y, z)| \\
&\lesssim 2^{-jn-km-\ell d}\int_{|x_1| \leq 2^{j}}
\frac{|x_1|}{|x|^{n+1}}
\phi^{(1)}(2^{-j}x_1) dx_1\\
&\lesssim \frac{2^{-jn-km-\ell d}}{|2^{-j}x|^{n+1}} \\
&\lesssim\frac{2^{-jn}}{(1 + 2^{-j} |x|)^{n + 1}}
\frac{2^{-km}}{(1 + 2^{-k} |y|)^{m + 1}}
\frac{2^{-\ell d}}{(1 + 2^{-\ell} |z|)^{d + 1}}.
\end{align*}

Since that {\bf Case 5} and {\bf Case 6:} $|x| \geq 3\cdot2^{j}, |y| <3\cdot2^{k}, |z| < 3\cdot2^{\ell}$
along with {\bf Case 7:} $|x| \geq 3\cdot2^{j}, |y| <3\cdot2^{k}, |z| < 3\cdot2^{\ell}$
are symmetric, so the proofs of these two cases follows from the one of {\bf Case 5}.

Now we deal with the last case.  {\bf Case 8:} $|x|<3\cdot2^{j}, |y| <3\cdot2^{k}, |z| < 3\cdot2^{\ell}$.
Let $\eta_1 \in C^\infty_0 (\mathbb{R}^n)$ with $0 \leq \eta_1(x) \leq 1$ and $\eta_1(x) = 1$ when $|x| \leq 4$, and $\eta_1(x) = 0$ when $|x| \geq 8$. We can define $\eta_2(y)$ on $\mathbb R^m$ and $\eta_2(z)$ on $\mathbb R^d$ similarly. Then
\begin{align*}
&|K\ast \phi_{j,k,\ell} (x,y,z)|\\
=& 2^{-jn-km-\ell d} \bigg| \int_{\mathbb R^d} \int_{\mathbb{R}^m} \int_{\mathbb{R}^n}
{K}(x_1, y_1,z_1)\phi^{(1)}(2^{-j}(x-x_1)) \phi^{(2)}(2^{-k}(y-y_1))\\
&\quad\times\phi^{(3)}
(2^{-\ell}(z-z_1)) \eta_1(2^{-j} x_1) \eta_2(2^{-k} y_1) \eta_3(2^{-\ell}  z_1)dx_1dy_1 dz_1 \bigg|\\
=:I+II,
\end{align*}
where
\begin{align*}
I=& 2^{-jn-km-\ell d} \bigg| \int_{\mathbb R^d} \int_{\mathbb{R}^m} \int_{\mathbb{R}^n}
{K}(x_1, y_1,z_1)(\phi^{(1)}(2^{-j}(x-x_1))-\phi^{(1)}(2^{-j}x))\\&\quad\times
(\phi^{(2)}(2^{-k}(y-y_1))-\phi^{(2)}(2^{-k}y))(\phi^{(3)}
(2^{-\ell}(z-z_1))-\phi^{(3)}(2^{-\ell}z))\\
&\quad\quad\times  \eta_1(2^{-j} x_1) \eta_2(2^{-k} y_1) \eta_3(2^{-\ell}  z_1)dx_1dy_1 dz_1 \bigg|
\end{align*}
and
\begin{align*}
II=& 2^{-jn-km-\ell d} \bigg| \int_{\mathbb R^d} \int_{\mathbb{R}^m} \int_{\mathbb{R}^n}
{K}(x_1, y_1,z_1)\phi^{(1)}(2^{-j}(x)) \phi^{(2)}(2^{-k}(y))\\
&\quad\times\phi^{(3)}
(2^{-\ell}(z)) \eta_1(2^{-j} x_1) \eta_2(2^{-k} y_1) \eta_3(2^{-\ell}  z_1)dx_1dy_1 dz_1 \bigg|.
\end{align*}

By applying the size condition on $\mathcal{K}$ and the smoothness condition on $\phi$ for $I$,
we conclude that
\begin{align*}
&|K \ast \phi_{j,k,\ell} (x,y,z)|\\
&\lesssim 2^{-jn-km-\ell d} \int_{|z_1|\leq 2^{\ell+3}}\int_{|y_1| \leq 2^{k+3}} \int_{|x_1| \leq 2^{j+3}} 
\frac{|2^{-j}x_1||2^{-k}y_1| 2^{-\ell}z_1|}{(|x_1|+|y_1| + |z_1|)^{n+m+d}}dx_1dy_1dz_1\\
&\lesssim 2^{-jn-km-\ell d}\\
&\lesssim\frac{2^{-jn}}{(1 + 2^{-j} |x|)^{n + 1}}
\frac{2^{-km}}{(1 + 2^{-k} |y|)^{m + 1}}
\frac{2^{-\ell d}}{(1 + 2^{-\ell} |z|)^{d + 1}}.
\end{align*}

Moreover, from the fact that the conditions of $K$ implies that $\widehat{K}$ is bounded,  
we obtain that
\begin{align*}
|K\ast \varphi_{j,k,\ell} (x,y,z)|
&\lesssim \int_{\mathbb R^d}\int_{\mathbb{R}^m} \int_{\mathbb{R}^n} \widehat{K} (\xi_1, \xi_2, \xi_3) \widehat\eta_1(2^{j} \xi_1) \widehat\eta_2(2^{k} \xi_2) \widehat\eta_3(2^{-\ell} \xi_3) d\xi_1 d\xi_2 d\xi_3\\
&\lesssim 2^{-jn-km-\ell d}\\
&\lesssim\frac{2^{-jn}}{(1 + 2^{-j} |x|)^{n + 1}}
\frac{2^{-km}}{(1 + 2^{-k} |y|)^{m + 1}}
\frac{2^{-\ell d}}{(1 + 2^{-\ell} |z|)^{d + 1}}.
\end{align*}

Combining $|\phi_{j,k,\ell} \ast {K} \ast \phi_{j',k',\ell'} (x, y, z)=|K\ast (\phi_{j,k,\ell}\ast \phi_{j',k',\ell'} (x, y, z))$ with the following classical almost orthogonal estimates: for a large positive integer $M$ we have
\begin{align*}
&\left|\phi_{j,k,\ell} \ast \phi_{j',k',\ell'} (x, y, z) \right|
\lesssim  2^{-|j-j'|M}2^{-|k-k'|M}2^{-|\ell-\ell'|M}\\
&\quad\times\frac{2^{-(j \vee j')n}}{(1 + 2^{-(j \vee j')} |x|)^{n+M}}
\frac{2^{-(k \vee k')m}}{(1 + 2^{-(k \vee k')} |y|)^{m+M}}
\frac{2^{-(\ell \vee \ell')d}}{(1 + 2^{-(\ell \vee \ell')} |z|)^{d+M}},
\end{align*}
we immediately obtain the desired result. Therefore, we prove this lemma.
\end{proof}

Now we recall the discrete Calder\'on reproducing identity on $L^2(\mathbb R^{n+m+d})\cap H^p_w(\mathbb R^n\times\mathbb R^m\times\mathbb R^d)$.
See \cite{HLLL,HC,Ru} for more details.
\begin{lemma}\label{DCI}
Let $w \in A_{\infty}(\mathbb R^n\times\mathbb R^m\times\mathbb R^d)$. Then there exists an operator $h\in L^2(\mathbb R^{n+m+d})\cap H^p_w(\mathbb R^n\times\mathbb R^m\times\mathbb R^d)$ such that 
for a sufficiently large $N\in \mathbb N$,
\[
f(x,y,z) = \sum_{j,k,\ell} \sum_{Q\in \mathcal{R}^{j-N,k-N,\ell-N}} |Q| {\phi}_{j,k,\ell}(x - x_I, y - y_J, z - z_K) \left( \phi_{j,k,\ell} \ast h \right)(x_I, y_J, z_K),
\]
where $0 < p < \infty$, and the series converges in $L^2(\mathbb{R}^{n+m+d})$.
Moreover,
$$
\|f\|_{L^2(\mathbb R^{n+m+d})}\sim \|h\|_{L^2(\mathbb R^{n+m+d})};\quad \|f\|_{H_w^p(\mathbb{R}^n\times \mathbb{R}^m \times \mathbb{R}^d)}
\sim \|h\|_{H_w^p(\mathbb{R}^n \times \mathbb{R}^m \times \mathbb{R}^d)}.
$$
\end{lemma}

The following lemma also plays a key role in the proof of Theorem \ref{1p-3},  which is a slight modification to \cite[Lemma 2.3]{Ru}.
Also see \cite{DH,FJ,HLW,Tan} for more details. 

\begin{lemma}\label{max}  
Given any integers \( j, k, \ell, j', k', \ell' \), let \( Q \in \mathcal{R}^{j,k,\ell} \) and 
\( Q' = I' \times J' \times K' \in \mathcal{R}^{j',k',\ell} \). Let \( \{a_{Q'}\} \) be any given sequence and 
let \( x_{Q'}^* = (x_{I'}^{*}, x_{J'}^{*}, x_{K'}^{*}) \) be any point in \( Q' \). 
Then for any \(u^* = (u^{*}_{1}, u^{*}_{2}, u^{*}_{3})\in Q \), 
\( v^* = (v^{*}_{1}, v^{*}_2, v^{*}_3) \in Q \), we have
\begin{align*}
&\sum_{Q' \in \mathcal{R}^{j',k',\ell}} \frac{2^{-(j \vee j')n} }
{\left(1+ 2^{-(j \vee j')}|u^{*}_1 -x_{I'}^{*}| \right)^{n+1}}  
\frac{ 2^{-(k \vee k')m}}
{\left(1 + 2^{-(k \vee k')}|u^{*}_2 - x_{J'}^{*} \right)^{m+1} }\\
&\qquad\qquad\qquad\qquad\times
\frac{2^{-(\ell \vee \ell')d}}
{\left(1 + 2^{-(\ell \vee \ell')}|u^{*}_3 - x_{K'}^{*}| \right)^{d+1} }
 |Q'| |a_{Q'}|\\
&\lesssim 2^{[n(j\vee j'-j')+m(k\vee k'-k')+d(\ell\vee \ell'-\ell')](\frac{1}{r}-1)}
\left\{ M_s \left(\sum_{Q' \in \mathcal{R}^{j',k',\ell'}} |a_{Q'}|^r \chi_{Q'} (v^*)\right)^{1/r}\right\},
\end{align*}
where \( (\frac{n}{n+1}\vee \frac{m}{m+1}\vee \frac{d}{d+1})< r \leq 1 \).
\end{lemma}

\noindent {\bf Proof of Theorem \ref{1p-3}}.\quad By applying Lemma \ref{DCI}, for $f\in L^2(\mathbb R^{n+m+d})\cap H^p_w(\mathbb R^n\times\mathbb R^m\times\mathbb R^d)$ we conclude that
$$ 
\begin{aligned}
&\left\|T(f)\right\|_{H^p_w(\mathbb R^n\times\mathbb R^m\times \mathbb R^d)}\\
&=\left\|\left\{\sum_{j, k, \ell \in \mathbb{Z}} \sum_{Q \in \mathcal{R}^{j-N, k-N,\ell-N}}\left|\left(\phi_{j, k,\ell} *K * f\right)\left(x_Q\right)\right|^2 \chi_Q\right\}^{1 / 2}\right\|_{L_w^p\left(\mathbb{R}^{n+m+d}\right)} \\
&= \Bigg\|\Bigg\{\sum_{j, k, \ell \in \mathbb{Z}} \sum_{Q \in \mathcal{R}^{j-N, k-N,\ell-N}}\Bigg|\phi_{j, k,\ell}\ast {K}\ast\Bigg(\sum_{j^{\prime}, k^{\prime, \ell^{\prime}} \in \mathbb{Z}} \sum_{Q^{\prime} \in \mathcal{R}^{j'-N, k'-N,\ell'-N}}\left|Q^{\prime}\right| \phi_{j^{\prime}, k^{\prime}, \ell^{\prime}}\ast
f\left(x_{R^{\prime}}\right)\\
&\qquad\qquad\qquad \times \phi_{j^{\prime}, k^{\prime}, \ell^{\prime}}(\cdot-x_{Q'})\Bigg)(x_Q)\Bigg|^2\chi_{Q}\Bigg\}^{1/2}\Bigg\|_{L^p_w(\mathbb R^{n+m+d})}\\
&= \Bigg\|\Bigg\{\sum_{j, k, \ell \in \mathbb{Z}} \sum_{Q \in \mathcal{R}^{j-N, k-N,\ell-N}}\Bigg|
\sum_{j^{\prime}, k^{\prime, \ell^{\prime}} \in \mathbb{Z}} \sum_{Q^{\prime} \in \mathcal{R}^{j'-N, k'-N,\ell'-N}}\left|Q^{\prime}\right| \phi_{j^{\prime}, k^{\prime}, \ell^{\prime}}\ast f\left(x_{R^{\prime}}\right)\\
&\qquad\qquad\qquad \times (\phi_{j, k,\ell}\ast {K}\ast \phi_{j^{\prime}, k^{\prime}, \ell^{\prime}})(x_Q-x_{Q'})\Bigg|^2\chi_{Q}\Bigg\}^{1/2}\Bigg\|_{L^p_w(\mathbb R^{n+m+d})}.
\end{aligned}
$$

Observe that $\tilde \gamma=\frac{n}{n+1}\vee\frac{m}{m+1}\vee \frac{d}{d+1}$.
 If $w\in A_{p/\tilde\gamma}(\mathbb R^n\times \mathbb R^m)$, then there exist $\tilde\gamma<r\le1$ such that
 $w\in A_{p/r}(\mathbb R^n\times \mathbb R^m)$.
%we choose $r<p$ such that $\frac{p}{r}>q_w$.
Thus, using Lemma \ref{ortho3}, Lemma \ref{max} and Proposition \ref{FS}, we find that
\begin{align*}
&\|T f\|_{H^p_w\left(\mathbb{R}^n\times\mathbb R^m\times\mathbb R^d\right)}\\
\lesssim& \left\|\left\{\sum_{j^{\prime}, k^{\prime}, \ell^{\prime} \in \mathbb{Z}}\left[\mathcal{M}_{s}
\left(\sum_{Q^{\prime} \in \mathcal{R}^{j'-N, k'-N,\ell'-N}}
\left|\left(\phi_{j^{\prime}, k^{\prime},\ell^{\prime}}\ast f\right)\left(x_{Q^{\prime}}\right)\right|^2 \chi_{Q^{\prime}}\right)^{r/2}\right]^{2/r}\right\}^{1/ 2}\right\|_{L_w^p\left(\mathbb{R}^{n+m+d}\right)}\\
\lesssim& \left\|\left\{\sum_{j^{\prime}, k^{\prime}, \ell^{\prime} \in \mathbb{Z}}
\sum_{Q^{\prime} \in \mathcal{R}^{j'-N, k'-N,\ell'-N}}
\left|\left(\phi_{j^{\prime}, k^{\prime},\ell^{\prime}}\ast f\right)\left(x_{Q^{\prime}}\right)\right|^2 \chi_{Q^{\prime}}\right\}^{1/ 2}\right\|_{L_w^p\left(\mathbb{R}^{n+m+d}\right)}\\
\sim& \|f\|_{H^p_w\left(\mathbb{R}^n\times\mathbb R^m\times\mathbb R^d\right)}.
\end{align*}

Thus, applying the fact that $L^2(\mathbb R^{n+m+d})\cap H^p_w(\mathbb R^n\times\mathbb R^m\times\mathbb R^d)$ is dense in $H^p_w(\mathbb R^n\times\mathbb R^m\times\mathbb R^d)$, we have completed the proof of Theorem \ref{1p-3}.
\hfill $\square$

\section{Proofs of Theorems \ref{fp} and \ref{fp-3}}
The proof of Theorem \ref{fp} is essentially the same as the proof of Theorem \ref{fp-3}.
Next we only need to give the proof of Theorem \ref{fp-3}. Before we prove it, first we recall the three-parameter singular integral operators. The three-parameter singular integral kernel $K$ is a distribution on
$\mathbb R^n\times \mathbb R^m\times \mathbb R^d$ which coincides with a smooth function away from the coordinate subspace $x = 0$, $y = 0$ and $z = 0$ and fulfills

\begin{itemize}
    \item[(i)] (differential inequalities) for each multi-index $\alpha = (\alpha_1, \dots, \alpha_n)$, $\beta = (\beta_1, \dots, \beta_m)$, $\gamma = (\gamma_1, \dots, \gamma_d)$, there exists a constant $C_{\alpha, \beta, \gamma} > 0$ such that
    \[
    \left| \partial_x^\alpha \partial_y^\beta \partial_z^\gamma \mathcal{K}(x, y, z) \right| \leq C_{\alpha, \beta, \gamma} |x|^{-n-|\alpha|} |y|^{-m-|\beta|} |z|^{-d-|\gamma|},
    \]
    \item[(ii)] (cancellation conditions) for every normalized bump function $\varphi_1$ on $\mathbb{R}^n$, $\varphi_2$ on $\mathbb{R}^m$ and $\varphi_3$ on $\mathbb{R}^d$ and for any $R_1, R_2, R_3 > 0$,
        \[
    \left| \int_{\mathbb{R}^d} \partial_x^\alpha \partial_y^\beta K(x, y, z)\varphi_3(R_3 z) \,dz \right| \leq C |x|^{-n-|\alpha|}|y|^{-m-|\beta|},
    \]
      \[
    \left| \int_{\mathbb{R}^m} \partial_x^\alpha \partial_z^\gamma K(x, y, z)\varphi_2(R_2y) \,dy \right| \leq C |x|^{-n-|\alpha|}|z|^{-d-|\gamma|},
    \]
    \[
    \left| \int_{\mathbb{R}^n} \partial_y^\beta \partial_z^\gamma K(x, y, z) \varphi_1(R_1 x)\, dx\right| \leq C 
    |y|^{-m-|\beta|}|z|^{-d-|\gamma|},
    \]
    \[
    \left| \int_{\mathbb{R}^m \times \mathbb{R}^d} \partial_x^\alpha K(x, y, z) \varphi_2(R_2 y) \varphi_3(R_3 z) \, dy \, dz \right| \leq C |x|^{-n-|\alpha|},
    \]
    \[
    \left| \int_{\mathbb{R}^n \times \mathbb{R}^d} \partial_y^\beta K(x, y, z) \varphi_1(R_1 x) \varphi_3(R_3 z) \, dx \, dz \right| \leq C |y|^{-m-|\beta|},
    \]
    \[
    \left| \int_{\mathbb{R}^n \times \mathbb{R}^m} \partial_z^\gamma K(x, y, z) \varphi_1(R_1 x) \varphi_2(R_2 y) \, dx \, dy \right| \leq C |z|^{-d-|\gamma|},
    \]
    \[
    \left| \int_{\mathbb{R}^n \times \mathbb{R}^m \times \mathbb{R}^d} K(x, y, z) \varphi_1(R_1 x) \varphi_2(R_2 y) \varphi_3(R_3 z) \, dx \, dy \, dz \right| \leq C.
    \]
\end{itemize}

Then the operator $T$ defined by $T(f)(x,y,z)=p.\,v.\, K\ast f(x,y,z)$ is said to three-parameter
singular integral operator. It was demonstrated in \cite{NRS} that flag kernels constitute a subclass of product kernels. Consequently, the boundedness of flag singular integrals on weighted product Hardy spaces follows as a corollary of the boundedness of three-parameter singular integral operators on these spaces. To end it, we only need to obtain the following three-parameter almost orthogonal
estimates for the three-parameter singular integral kernels.
The rest of the proof of Theorem \ref{fp-3} is identical to Section 4.

\begin{proposition}\label{ortho4} 
Suppose that $\phi$ is defined in Section 4 and $K$ is a three-parameter singular integral kernel as above. Then for any given positive integers $L$ and $M$ we have
\begin{align*}
&\left|\phi_{j,k,\ell} \ast {K} \ast \phi_{j',k',\ell'} (x, y, z) \right|\lesssim  2^{-|j-j'|L}2^{-|k-k'|L}2^{-|\ell-\ell'|L}
\frac{2^{-(j \vee j')n}}{(1 + 2^{-(j \vee j')} |x|)^{n+M}}\\
&\qquad\qquad\times\frac{2^{-(k \vee k')m}}{(1 + 2^{-(k \vee k')} |y|)^{m+M}}
\frac{2^{-(\ell \vee \ell')d}}{(1 + 2^{-(\ell \vee \ell')} |z|)^{d+M}},
\end{align*}
for all $(x,y,z)\in \mathbb R^n\times\mathbb R^m\times \mathbb R^d$.
\end{proposition}

\begin{proof}
First we recall the classical almost orthogonal estimates for one-parameter singular integral kernel:
Let $\mathcal K$ be a one-parameter inhomogeneous singular integral kernel on $\mathbb R^n$ and
$\phi$ is a Schwartz functions supported in the unit balls centered at the origin fulfilling the sufficiently high moment conditions. Then for any fixed positive integers $L$ and $M$ we have
\begin{align}\label{ortho5}
\left|\phi_{j} \ast \mathcal{K} \ast \phi_{j'} (x) \right|\lesssim  2^{-|j-j'|L}\frac{2^{-(j \vee j')n}}{(1 + 2^{-(j \vee j')} |x|)^{n+M}},
\end{align}
for all $x\in \mathbb R^n$.
We write 
\begin{align*}
&\phi_{j,k,\ell} \ast {K} \ast \phi_{j',k',\ell'} (x, y, z)\\
&=\iint\iint\iint \phi_j^{(1)}(x-x_1)\phi_k^{(2)}(y-y_1)
\phi_\ell^{(3)}(z-z_1)K(x_2,y_2,z_2)\\
&\quad\quad\times\phi_{j'}^{(1)}(x_1-x_2)\phi_{k'}^{(2)}(y_1-y_2)
\phi_{\ell'}^{(3)}(z_1-z_2)dx_1dx_2dy_1dy_2dz_1dz_2.
\end{align*}
Denote
\begin{align*}
K_{2,3}(y_2,z_2)
&=\iint \phi_j^{(1)}(x-x_1)K(x_2,y_2,z_2)\phi_{j'}^{(1)}(x_1-x_2)dx_1dx_2\\
&=:\iint \phi_j^{(1)}(x-x_1)\tilde K(y_2,z_2)(x_2)\phi_{j'}^{(1)}(x_1-x_2)dx_1dx_2.
\end{align*}
For fixed $y_2$ and $y_3$, we can find that $\tilde K(y_2,z_2)(x_2)$
is a one-parameter singular integral kernel.
Therefore, applying (\ref{ortho5}) yields that
\begin{align*}
\left|K_{2,3}(y_2,z_2) \right|\lesssim_{y_2,z_2}  2^{-|j-j'|L}\frac{2^{-(j \vee j')n}}{(1 + 2^{-(j \vee j')} |x|)^{n+M}}.
\end{align*}
Then repeating similar but easier argument to \cite[Lemma 3.3]{HC},
we know that the kernel $K_{2,3}(y_2,z_2)$ is a two-parameter singular integral
kernel satisfying the regularity condition and the cancellation conditions.
Moreover, let
\begin{align*}
K_3(z_2)
&=\iint \phi_k^{(2)}(y-y_1)
K_{2,3}(y_2,z_2)
\phi_{k'}^{(2)}(y_1-y_2)
dy_1dy_2.
\end{align*}
Similarly, we find that $K_3(z_2)$ is s a one-parameter singular kernel with 
\begin{align*}
&|K_3(z_2)|\lesssim_{z_2}  2^{-|j-j'|L}2^{-|k-k'|L}
\frac{2^{-(j \vee j')n}}{(1 + 2^{-(j \vee j')} |x|)^{n+M}}
\frac{2^{-(k \vee k')m}}{(1 + 2^{-(k \vee k')} |y|)^{m+M}}.
\end{align*}
Applying (\ref{ortho5}) again yields that
\begin{align*}
&\iint
\phi_\ell^{(3)}(z-z_1)K_3(z_2)
\phi_{\ell'}^{(3)}(z_1-z_2)dz_1dz_2\\
&\lesssim  2^{-|j-j'|L}2^{-|k-k'|L}2^{-|\ell-\ell'|L}
\frac{2^{-(j \vee j')n}}{(1 + 2^{-(j \vee j')} |x|)^{n+M}}\\
&\qquad\qquad\times\frac{2^{-(k \vee k')m}}{(1 + 2^{-(k \vee k')} |y|)^{m+M}}
\frac{2^{-(\ell \vee \ell')d}}{(1 + 2^{-(\ell \vee \ell')} |z|)^{d+M}},
\end{align*}
which prove this proposition.
\end{proof}

\section{A passage to ball Banach function spaces}
In this section, we will conclude by briefly considering the extension of our main results
to Hardy spaces for ball quasi-Banach function spaces.
In the case of one-parameter function spaces, Sawano et al. \cite{SHYY17} introduced the ball quasi-Banach function space $X$ and established the real-variable theory for the corresponding one-parameter Hardy spaces.
For more details, see \cite{tanzhang,wang2020applications,zhang2021weak} and the reference therein.
On the another hand, the author \cite{Tan1} recently introduced and studied the multi-paramter Hardy space ${H}_X$ associated with the ball quasi-Banach function space $X$. These results can be applied to various concrete examples of ball quasi-Banach function spaces, including product Herz spaces \cite{W}, weighted product Morrey spaces \cite{W1}, and product Orlicz spaces \cite{Tan1}. 

Let $X(\mathbb R^n\times \mathbb R^m)$ be a ball quasi-Banach function space.
The three-parameter Hardy spaces $H_X(\mathbb R^n\times\mathbb R^m\times \mathbb R^d)$ are the set of all distributions $f$
for which the quantity
$$\|f\|_{{H}_X(\mathbb R^n\times\mathbb R^m\times \mathbb R^d)}=\|\mathcal G^d(f)\|_{X(\mathbb R^{n+m+d})}<\infty.$$
Under some mild assumptions, the three-parameter Hardy spaces $H_X(\mathbb R^n\times\mathbb R^m\times \mathbb R^d)$ 
is well defined in terms of the discrete product Littlewood--Paley theory and the Min-Max type inequalities. For more details, we refer to \cite{Tan,Tan1}. 

Therefore, we point out that the results in this paper can be extended to multi-parameter Hardy spaces associated with ball quasi-Banach function spaces ${H}_X(\mathbb R^n\times\mathbb R^m \times \mathbb R^d)$, where $X$ is a ball quasi Banach function space.
Repeating the nearly identical argument to  \cite{BTZ,Tan,Tan1}, we can establish the boundedness of the one-parameter singular integral operator and flag singular integrals on three-parameter Hardy space ${H}_X(\mathbb R^n\times\mathbb R^m \times \mathbb R^d)$ with the help of the weighted Hardy spaces estimates and the refined Rubio de Francia extrapolation. 
The results can then be applied to specific examples of ball quasi-Banach function spaces, such as product Herz spaces and weighted product Morrey spaces. %To our best knowledge, all the obtained results on these Hardy-type spaces are entirely new.

\subsection*{Acknowledgements}
The author is supported by the National Natural Science Foundation of China (Grant No.11901309), Natural Science Foundation of Nanjing University of Posts and Telecommunications (Grant No.NY224167) and the Open Project Program of Key Laboratory of Mathematics and Complex System (Grant No. K202502), Beijing Normal University.
%the Jiangsu Government Scholarship for Overseas Studies.
%The author also wish to express their heartfelt thanks to the anonymous referees for many corrections and so valuable suggestions that improved the paper significantly.

%{\bf Acknowledgments.}
%The authors wish to express her heartfelt thanks to the anonymous referees for many corrections and valuable suggestions that improved the paper significantly. The author is supported by the the Jiangsu Government Scholarship for Overseas Studies and the National Natural Science Foundation of China (Grant No. 11901309). %He would like to express his gratitude to U for their pleasant hospitality.
%, the National Natural Science Foundation of China (Grant No. 11901309), the China Postdoctoral Science Foundation (Grant No. 2023T160296) and the Natural Science Foundation of Nanjing University of Posts and Telecommunications (Grant No. NY222168).

%%%%%%%%%%% To ease editing, use normal size for the references:

\end{document}